\documentclass[a4paper,12pt]{article}

\usepackage{amsmath, amssymb, xspace, epic, eepic, cite}
\usepackage{gunther2e}

\def\rank{\mathop{\rm rk}\nolimits}
\def\termrank{\mathop{\rm rk_t}\nolimits}

\date{13th February 2004}

\title{On the maximum rank of Toeplitz block matrices of blocks of a given pattern}

\author{Gunther Rei\ss ig%
\thanks{%
Otto-von-Guericke-Universit\"at Magdeburg, %
Chair of Systems Theory, %
Institut f\"ur Automatisierungstechnik (FEIT-IFAT), %
PF 4120, %
D-39016 Magdeburg, %
Germany.
You may find a {Bib\TeX} entry for this paper at
http://www.reiszig.de/gunther/%
}\;
\thanks{This is the accepted version of a paper
  published in Proc. 16th Intl. Symp. on Math. Th. of Networks
   and Systems (MTNS), Leuven, Belgium, July 5-9, 2004,
   B. de Moor, B. Motmans, J. Willems, P. Van Dooren,
   and V. Blondel, editors.
   ISBN 90-5682-517-8.
\copyright Katholieke Universiteit Leuven, Departement Elektrotechniek
(ESAT), Heverlee, Belgium, 2004. Some errors have been corrected.}
}

\begin{document}
\maketitle
\begin{abstract}
We show that the maximum rank of block lower triangular Toeplitz
block matrices equals their term rank if the blocks fulfill a
structural condition, i.e., only the locations but not the values of
their nonzeros are fixed.
\end{abstract}
\section{Introduction}
\label{sec:intro}
Associated with a sequence $H$ of matrices,
\[
H=\left( H_i \right)_{i\in\mathbb{N}\cup\{0\}},
\]
or, equivalently, with a formal Laurent series $H$,
\[
H(s)=\sum_{i=0}^{\infty} s^{-i}H_{i},
\]
there are Toeplitz block matrices
\[
T_{k+1}(H)=
\left(
\arraycolsep.1em
\begin{array}{lccc}
H_0\\
H_{1} & H_0\\
\vdots & \ddots & \ddots\\
H_{k} & \cdots & H_{1} & H_0
\end{array}
\right),
k\geq0,
\]
which play a part in systems and control theory
\cite{VandewalleDewilde74,PughJonesDemianczukHayton89,MattsonSoederlind93,Wiedemann98,i02inconsistent}
and are also related to problems in other fields, e.g.
\cite{Forney75,Poljak89}.

In the structural approach to the analysis of linear systems, initiated by
Iri, Tsunekawa, Yajima \cite{IriTsunekawaYajima71} and Lin \cite{Lin74b},
one assumes that each nonzero of the matrices by which the system is
described is either a fixed zero or a free parameter
\cite{YamadaFoulds90,DionCommaultvanderWoude03}, i.e.,
the following \begriff{structural condition} is imposed:
\begin{equation}
\label{e:StrCondition}
\begin{minipage}{11cm}
Each nonzero of the matrices equals some parameter,
and none of these parameters appears more than once in the matrices.
\end{minipage}
\end{equation}
For such parameter dependent systems, one asks for \begriff{generic}
(or \begriff{structural} or \begriff{typical}) properties of systems,
that is, properties that the system has for almost all (in one sense
or another) values of the parameters
\cite{YamadaFoulds90,DionCommaultvanderWoude03}.
To see the relation to maximum ranks of Toeplitz block matrices,
assume the matrices
$H_i$ depend on some parameter $p\in\mathbb{R}^q$,
\[
H_i\colon\mathbb{R}^q\to\mathbb{R}^{n\times m}
\colon p\mapsto H_i(p) \text{\ \ for all $i\in\mathbb{N}\cup\{0\}$},
\]
so that $T_{k}(H)$ depends on $p$ as well,
\[
T_{k}(H)\colon \mathbb{R}^q\to\mathbb{R}^{nk\times mk}.
\]
If $H$ fulfills a \begriff{structural condition}, i.e., the parameter
dependent matrix
\begin{equation}
\label{e:StrConditionLaurent}
p\mapsto(H_{k-1}(p),\dots,H_{1}(p),H_0(p))
\text{\ \ \ fulfills \ref{e:StrCondition}},
\end{equation}
then the dependence of $T_{k}(H)$ on
$p$ is analytic. This implies that the generic variants of
system properties characterizable by ranks of Toeplitz block matrices
and their submatrices are
characterizable by the maxima over $p\in\mathbb{R}^q$ of the ranks of
$T_{k}(H)(p)$ and its submatrices.
The difficulty is that even though $H$ fulfills \ref{e:StrConditionLaurent},
$T_{k}(H)$ does not fulfill \ref{e:StrCondition}, so that it may be
very hard to determine maximum ranks if $n$, $m$ or $k$ is large.

It is the purpose of this paper to show that for all $k\geq 1$ the
maximum rank of $T_k(H)$ equals its term rank, i.e., that the obvious
inequality
\begin{equation}
\label{e:inequal}
\max_{p\in\mathbb{R}^q}\,\rank T_k(H)(p)
\leq
\max_{p_1,\dots,p_{K}\in\mathbb{R}^q} \rank
\left(
\arraycolsep.1em
\def\arraystretch{2}
\begin{array}{lccc}
H_0(p_1)\\
H_{1}(p_{k+1}) & H_0(p_2)\\
\vdots & \ddots & \ddots\\
H_{k-1}(p_{K}) & \cdots & H_{1}(p_{2k-1}) & H_0(p_k)
\end{array}
\right)
\end{equation}
is actually an equality, provided that $H$ fulfills the structural
condition \ref{e:StrConditionLaurent}.
This result holds for matrices $H_i$
over an arbitrary field.
(Here, $K=k(k-1)/2$ and $\rank X$ denotes the rank of $X$.
The right hand side of \ref{e:inequal} is called the
\begriff{term rank} of $T_k(H)$ and is denoted by $\termrank T_k(H)$.)

Special cases of this result have been known for some time.
The case $k=1$ is due to Edmonds \cite{Edmonds67}. According to
\cite{Poljak89}, the case $k=2$
had been solved by Holenda and Schlegel in 1987 under the assumption that
$H_{1}$ is diagonal and nonsingular. Equality has been shown in
\cite{Poljak89} for general $k$ under the assumptions that $H_{1}$ is
diagonal and nonsingular and that $H$ is a pencil, i.e.,
$H(p)(s)=H_0(p)+s^{-1}H_{1}(p)$.
Finally, equality has been shown in \cite{i98MTNS,Wiedemann98} for general
$k$ under the assumption that $H$ is a pencil.
In addition, an analogous result for one particular submatrix of
$T_{n+1}(H)$ and some particular form of $H$ has been obtained in
\cite{ShieldsPearson76}, which is not a special case of the result of
this paper.

The proof we present in section \ref{section:main} is
elementary and uses only two nontrivial facts, namely, the
K\"onig-Egerv{\'a}ry Theorem and a simple lemma from parametric
programming. Moreover, compared to the proofs of those special cases
of our result that have been obtained earlier, our proof must also be
called extremely short.
\section{Preliminaries}
\label{section:Preliminaries}
We introduce some convenient notation and collect some well-known
facts.

If $R$ and $C$ are finite sets and $\mathbb{F}$ is a field,
we call any mapping $M\colon R\times C\to\mathbb{F}$ a
\textit{matrix of size $R\times C$ over $\mathbb{F}$}.
The entry in \textit{row} $r\in R$ and \textit{column} $c\in C$ of $M$
is denoted by $M_{r,c}$.
By a \textit{formal Laurent series of size $R\times C$ over $\mathbb{F}$} we mean a
sequence $H$,
\[
H=(H_{i})_{i\in\mathbb{N}\cup\{0\}},
\]
where the coefficients $H_{i}$ of $H$ are matrices of size $R\times C$ over
$\mathbb{F}$.

$T_k(H)$ denotes the block lower triangular Toeplitz block matrix
associated with $H$, which is the matrix of size
$(\{1,\dots,k\}\times R)\times(\{1,\dots,k\}\times C)$ over $\mathbb{F}$
defined by
\[
T_k(H)_{(i,r),(j,c)}=
\begin{cases}
0,&\text{if $i<j$,}\\
(H_{i-j})_{r,c},&\text{otherwise}.
\end{cases}
\]

Let $H$ be a Laurent series of size $R\times C$ over $\mathbb{F}$ and
assume $R\cap C=\emptyset$. The bipartite graph $G(H)$ associated with
$H$ is defined by $G(H)=(R,C,E)$, where
\[
E=\{\{r,c\}\,|\,r\in R,\,c\in C,\,(H_{i})_{r,c}\not=0\text{\ for some
  $i\in\mathbb{N}\cup\{0\}$}\}.
\]
Likewise, the weight function $w\colon E\to-\mathbb{N}\cup\{0\}$
associated with $H$ is defined by
\[
w_{r,c}=-\min\{i\in\mathbb{N}\cup\{0\}\,|\,(H_i)_{r,c}\not=0\}.
\]

Let $G$ be a bipartite graph,
$G=(R,C,E)$ with $R\cap C=\emptyset$, with vertex set $R\cup C$
and edge set
$
E\subseteq\{\{r,c\}\,|\,r\in R,\,c\in C\},
$
and let
\[
A\colon (R\cup C)\times E\to\{0,1\}
\]
be its incidence matrix \cite{Recski89}.
Let further $w\colon E\to\mathbb{Z}$ be a weight function on the edge
set of $G$. In the following,
\begin{align*}
&X\colon E\to\{0,1\},\\
&y\colon R\to\mathbb{N}\cup\{0\},\\
&z\colon C\to\mathbb{N}\cup\{0\}.
\end{align*}
The bipartite cardinality matching problem $B(G)$ is to
\begin{align}
\tag{$B(G)-a$}
\text{maximize}&\sum_{\{r,c\}\in E}X_{r,c}\\
\tag{$B(G)-b$}
\text{s.t.\ }&
AX\leq 1.
\end{align}
Its dual, the bipartite covering problem $DB(G)$, is to
\begin{align}
\tag{$DB(G)-a$}
\text{minimize}&\sum_{r\in R}y_r + \sum_{c\in C}z_c\\
\tag{$DB(G)-b$}
\text{s.t.\ }&
A^T\begin{pmatrix}y\\z\end{pmatrix}\geq 1.
\end{align}
The bipartite cardinality-$\mu$ assignment problem $A(G,w,\mu)$ is to
\begin{align}
\tag{$A(G,w,\mu)-a$}
\text{maximize}&\sum_{\{r,c\}\in E}X_{r,c}w_{r,c}\\
\tag{$A(G,w,\mu)-b$}
\text{s.t.\ }&
AX\leq 1,\\
\tag{$A(G,w,\mu)-c$}
&\sum_{\{r,c\}\in E}X_{r,c}=\mu.
\end{align}
Its dual, $DA(G,w,\mu)$, is to
\begin{align}
\tag{$DA(G,w,\mu)-a$}
\text{minimize\ }&\lambda\mu+\sum_{r\in R}y_r + \sum_{c\in C}z_c\\
\tag{$DA(G,w,\mu)-b$}
\text{s.t.\ }&
A^T\begin{pmatrix}y\\z\end{pmatrix}+\lambda\geq w.
\end{align}
By the K\"onig-Egerv{\'a}ry Theorem
\cite{LovaszPlummer86},
the optimal values of
$B(G)$ and $DB(G)$ coincide, and those of $A(G,w,\mu)$ and
$DA(G,w,\mu)$ coincide as well.\\
We say that $X$ (resp., $(y,z,\lambda)$) is \begriff{admissible} for problem
(\stern), if $X$ (resp., $(y,z,\lambda)$) fulfills $(\stern-b)$ and, if
present, $(\stern-c)$. We say that $X$ (resp., $(y,z,\lambda)$) is \begriff{optimal}
for problem (\stern) if it is a solution of problem (\stern).\\
The following fact \cite{RoosTerlakyVial97}
is also useful.
\begin{lemma}
\label{lemma:SubdifferentialConditionForLambda}
Let $w$ be nonpositive, $\hat \mu$ the optimal value of $B(G)$,
$\mu\in[0,\hat\mu]\cap\mathbb{Z}$, and 
$\delta(\mu)$ the optimal value of $A(G,w,\mu)$.\\
Then, for given $\lambda\in\mathbb{Z}$,
there are $y\in\mathbb{Z}^R$ and $z\in\mathbb{Z}^C$
such that $(y,z,\lambda)$ is a solution of $DA(G,w,\mu)$ iff the
following two conditions hold:
\begin{enumerate}
\item
$\mu>0$ $\Longrightarrow$ $\lambda\leq\delta(\mu)-\delta(\mu-1)$,
\item
$\mu<\hat\mu$ $\Longrightarrow$ $\lambda\geq\delta(\mu+1)-\delta(\mu)$.
\end{enumerate}
\end{lemma}
\section{Results}
\label{section:main}
Unless stated otherwise,
$H$ is a Laurent series of size $R\times C$ over $\mathbb{F}$.
If $R\cap C=\emptyset$,
$G$ denotes the bipartite graph associated with $H$, $G=G(H)=(R,C,E)$,
$w\colon E\to\mathbb{Z}$ denotes the weight function associated with $H$,
$X\in\{0,1\}^{R\times C}$,
$y\in\mathbb{Z}^R$, $z\in\mathbb{Z}^C$, $\lambda\in\mathbb{Z}$,
$X'\in\{0,1\}^{(\{1,\dots,k\}\times R)\times(\{1,\dots,k\}\times C)}$,
$y'\in\{0,1\}^{\{1,\dots,k\}\times R}$, and
$z'\in\{0,1\}^{\{1,\dots,k\}\times C}$.\\
We also assume $X_{r,c}=0$ if $\{r,c\}\notin E$ and
$X'_{(i,r),(j,c)}=0$ if $\{(i,r),(j,c)\}$ is not an edge of $G(T_k(H))$.

Our method is to transform primal and dual solutions of the
cardinality-$\mu$ assignment problem for $G$ and $w$ into primal and
dual solutions, respectively, of the matching problem for
$G(T_k(H))$. To this end, we define $X'$, $y'$ and $z'$ for any
$k\in\mathbb{N}$ by
\begin{subequations}\label{e:FromZOPtoB}
\begin{align}
\label{e:FromZOPtoB:a}
X'_{(i,r),(j,c)}&=
\begin{cases}
1,& \text{if $X_{r,c}=1$ and $w_{r,c}=j-i$},\\
0,& \text{otherwise,}
\end{cases}\\
\label{e:FromZOPtoB:b}
y'_{i,r}&=
\begin{cases}
1,& \text{if $i\geq 1-y_r-\lambda$},\\
0,& \text{otherwise,}
\end{cases}\\
\label{e:FromZOPtoB:c}
z'_{j,c}&=
\begin{cases}
1,& \text{if $z_c\geq j$},\\
0,& \text{otherwise,}
\end{cases}
\end{align}
\end{subequations}
where $i,j\in\{1,\dots,k\}$, $r\in R$ and $c\in C$.
\begin{proposition}
\label{proposition:FromZOPtoB}
Let $\mu\in\mathbb{N}\cup\{0\}$, $k\in\mathbb{N}$, and assume that
$X$ is admissible for $A(G,w,\mu)$,
$(y,z,\lambda)$ is admissible 
for $DA(G,w,\mu)$ and that \ref{e:FromZOPtoB} holds.\\
Then $X'$ is admissible for $B(G(T_k(H)))$ and
$(y',z')$ is admissible for $DB(G(T_k(H)))$.\\
If, in addition, the condition
\[
(\mu=0 \;\vee\; \lambda\geq-k)\;\wedge\;(\mu=|R| \;\vee\; \lambda=-k)
\]
holds and
$X$ and $(y,z,\lambda)$ are optimal, then so are $X'$ and $(y',z')$.
\end{proposition}
\begin{proof}
Assume $\mu>0$ without loss.\\
It is obvious from \ref{e:FromZOPtoB:a} that $X'_{(i,r),(j,c)}\geq
0$ for all $i,j\in\{1,\dots,k\}$, $r\in R$ and $c\in C$. Further,
$X'_{(i,r),(j,c)}=1$ implies $\{r,c\}\in E$ and
$w_{r,c}=j-i$, so that $(H_{i-j})_{r,c}\not=0$ and
$\{(i,r),(j,c)\}$ is an edge of $G(T_k(H))$.\\
If $X'_{(i,r),(j,c)}=X'_{(i,r),(j',c')}=1$, then
$X_{r,c}=X_{r,c'}=1$ by \ref{e:FromZOPtoB:a}, and hence,
$c=c'$ by admissibility of $X$. Further, $j-i=w_{r,c}=w_{r,c'}=j'-i$
by \ref{e:FromZOPtoB:a}, which implies $j=j'$.
Analogously, $X'_{(i,r),(j,c)}=X'_{(i',r'),(j,c)}=1$ implies $r=r'$
and $i=i'$, and thus, $X'$ is admissible for $B(G(T_k(H)))$.\\
Likewise,
$y'_{i,r}\geq0$ and $z'_{j,c}\geq0$ for all
$i,j\in\{1,\dots,k\}$, $r\in R$ and $c\in C$
by \ref{e:FromZOPtoB:b} and \ref{e:FromZOPtoB:c}.
If $\{(i,r),(j,c)\}$ is an edge of $G(T_k(H))$
and $y'_{i,r}+z'_{j,c}=0$, then $\{r,c\}$ is an edge of $G$,
$w_{r,c}\geq j-i$, and $y'_{i,r}=z'_{j,c}=0$. Therefore,
$i\leq-y_r-\lambda$ and $z_c\leq j-1$ from \ref{e:FromZOPtoB:b} and
\ref{e:FromZOPtoB:c}, and hence, $y_r+z_c+\lambda\leq j-i-1\leq
w_{r,c}-1$. This is a contradiction, as $(y,z,\lambda)$ is admissible
for $DA(G,w,\mu)$. Thus, $(y',z')$ is admissible for $DB(G(T_k(H)))$.

Now let $X$ and $(y,z,\lambda)$ be optimal, let $\hat \mu$ be the
optimal value of $B(G)$, and for every
$\mu\in[0,\hat\mu]\cap\mathbb{Z}$,
denote by $\delta(\mu)$
the optimal value of $A(G,w,\mu)$, particularly,
\[
\delta(\mu)=\sum_{\{r,c\}\in E} w_{r,c}X_{r,c}=\lambda\mu+\sum_{r\in
  R}y_r+\sum_{c\in C}z_c.
\]
If $\xi_{-i}$ is the number of
edges $\{r,c\}\in E$ for which $X_{r,c}=1$ and $w_{r,c}=-i$, then
\begin{align}
\label{e:proof1:1}
\delta(\mu)
&=
\sum_{i=0}^{\infty} (-i)\xi_{-i},\\
\label{e:proof1:2}
\mu
&=
\sum_{i=0}^{\infty}\xi_{-i}.
\end{align}
Observe that $i>\delta(\mu-1)-\delta(\mu)$ implies $\xi_{-i}=0$, for
if $\xi_{-i}>0$, then $X_{r_0,c_0}=1$ and
$w_{r_0,c_0}=-i<\delta(\mu)-\delta(\mu-1)$
for some edge $\{r_0,c_0\}\in E$. If we define $\tilde X$ by
$\tilde X_{r_0,c_0}=0$ and $\tilde X_{r,c}=X_{r,c}$ for
$\{r,c\}\not=\{r_0,c_0\}$, then
\[
\delta(\mu-1)
\geq
\sum_{\{r,c\}\in E} \tilde
X_{r,c}w_{r,c}=\delta(\mu)-w_{r_0,c_0}>\delta(\mu-1),
\]
which is a contradiction.\\
Hence, the sums in \ref{e:proof1:1} and \ref{e:proof1:2} actually
extend over $i=0$ to $k$, so that
\begin{align}
\notag
\sum_{(i,r),(j,c)} X'_{(i,r),(j,c)}
&=
\delta(\mu)-\sum_{i=0}^k(-i)\xi_{-1}+\sum_{i=0}^k(k-i)\xi_{-i}\\
&=\delta(\mu)+k\mu.
\label{e:proof1:*}
\end{align}
It remains to show the identity
\begin{equation}
\label{e:proof1:**}
\delta(\mu)+k\mu
=
\sum_{r\in R}\sum_{i=1}^ky'_{i,r}
+
\sum_{c\in C}\sum_{j=1}^kz'_{j,c},
\end{equation}
as the optimality of both $X'$ and $(y',z')$ follows from
\ref{e:proof1:*}, \ref{e:proof1:**} and the K\"onig-Egerv{\'a}ry
Theorem.\\
First, our assumption $\lambda\geq -k$ implies $\lambda+k+y_r\geq 0$,
from which
$
\sum_{i=1}^ky'_{i,r}=\lambda+k+y_r
$
follows by \ref{e:FromZOPtoB:b}, and $\sum_{j=1}^kz'_{j,c}=z_c$
follows directly from \ref{e:FromZOPtoB:c}. Hence, the value of the
right hand side of \ref{e:proof1:**} equals
\begin{align}
\nonumber
\sum_{r\in R}(\lambda+k+y_r)
+
\sum_{c\in C}z_c
&=
|R|(k+\lambda)+\sum_{r\in R}y_r +\sum_{c\in C}z_c\\
&=
\label{proposition:FromZOPtoB:proof:lastequ}
\delta(\mu)+|R|(k+\lambda)-\lambda\mu.
\end{align}
Finally, the value \ref{proposition:FromZOPtoB:proof:lastequ}
equals that of the left hand side of  \ref{e:proof1:**},
as $|R|=\mu$ or $\lambda=-k$ by assumption.
\end{proof}
\begin{theorem}
\label{theorem:TermRankEqualsGenericRank}
Let $R$ and $C$ be finite sets and $k\in\mathbb{N}$, let
$H$ be a formal Laurent series of size $R\times C$ over a field
$\mathbb{F}$, and let the coefficients of $H$ depend on some parameter
$p\in\mathbb{F}^q$.\\
If $H$ fulfills the structural condition \ref{e:StrConditionLaurent},
then the maximum over $p\in\mathbb{F}^q$ of the rank of $T_k(H)(p)$
equals the term rank of $T_k(H)$,
\[
\max_{p\in\mathbb{F}^q}\, \rank T_k(H)(p)
=
\termrank T_k(H).
\]
Moreover, the maximum is attained for some $p\in\{0,1\}^q$ and equals
the number of nonzeros in $T_k(H)(p)$ for that value of $p$.
\end{theorem}
\begin{proof}
We may assume without loss that $R$ and $C$ are nonempty
and disjoint.\\
Let $G$ and $w$ be the bipartite graph and the weight function,
respectively, associated with $H$,
$G=(R,C,E)$ and $w\colon E\to-\mathbb{N}\cup\{0\}$.\\
Let further $\hat\mu$ be the optimal value of $B(G)$, i.e., the
maximum cardinality of a matching in $G$, and for
every $\mu\in[0,\hat\mu]\cap\mathbb{Z}$,
denote by $\delta(\mu)$ the optimal value of $A(G,w,\mu)$.\\
As $\delta$ is concave, there exists some $\mu$ and a solution $(y,z,\lambda)$
of $DA(G,w,\mu)$ with $\lambda=-k$ by Lemma
\ref{lemma:SubdifferentialConditionForLambda}.
Prop. \ref{proposition:FromZOPtoB} implies that for any solution $X$
of $A(G,w,\mu)$, the matching $X'$ defined by \ref{e:FromZOPtoB:a} is a
maximum matching in $G(T_k(H))$, in other words,
$\termrank T_k(H)=\rank X'$.\\
Now choose a special value of the parameter $p$ by setting
\[
p_{i,r,c}=
\begin{cases}
1,&\text{if $w_{r,c}=-i$ and $X_{r,c}=1$},\\
0,&\text{otherwise}
\end{cases}
\]
for all parameter components $p_{i,r,c}$ associated with a nonzero
$(H_{i})_{r,c}$. Obviously, $T_k(H)(p)=X'$ for this particular $p$,
so that $\rank T_k(H)(p)=\termrank T_k(H)$.
\end{proof}
\section{Conclusions}
We have shown that the maximum rank of block lower triangular Toeplitz
block matrices equals their term rank if the blocks fulfill a
structural condition, i.e., only the locations but not the values of
their nonzeros are fixed. This result holds for matrices over an
arbitrary field.

The proof we presented is elementary and, compared to the proofs of
those special cases that have been obtained earlier, it is also
extremely short. Further results related to the structural
approach to the analysis of linear state space and descriptor systems
that we have obtained using the same techniques will be detailed
elsewhere.
\small
\bibliographystyle{unsrtabb}
\bibliography{03}
\end{document}